\documentclass[a4paper,10pt, leqno]{amsart}
\usepackage{graphicx}
\usepackage{amssymb}

\numberwithin{equation}{section}
\newtheorem{Theorem}{Theorem}[section]

\newtheorem{Fact}[Theorem]{Fact}

\newtheorem{Proposition}[Theorem]{Proposition}

\theoremstyle{remark}
\newtheorem{Remark}[Theorem]{Remark}
\newtheorem{Rmk}[Theorem]{Remark}
\theoremstyle{definition}
\newtheorem{Definition}[Theorem]{Definition}

\newtheorem{Example}[Theorem]{Example}
\newtheorem{Exa}[Theorem]{Example}
\newtheorem*{acknowledgements}{Acknowledgements}

\newcommand{\R}{{\mathbb R}}

\newcommand{\Z}{{\mathbb Z}}

\newcommand{\E}{{\mathbb E}}
\newcommand{\Lo}{{\mathbb L}}
\newcommand{\mc}[1]{{\mathcal #1}}
\newcommand{\mb}[1]{{\mathbf #1}}
\newcommand{\pmt}[1]{{\begin{pmatrix} #1  \end{pmatrix}}}

\renewcommand{\phi}{\varphi}
\renewcommand{\epsilon}{\varepsilon}
\newcommand{\op}[1]{{\operatorname{ #1}}}

\renewcommand{\phi}{\varphi}
\renewcommand{\epsilon}{\varepsilon}

\renewcommand{\phi}{\varphi}

\title{
Umbilics of surfaces \\
in the Lorentz-Minkowski $3$-space
}

\author{Naoya Ando}
\address[Naoya Ando]{%
Department of Mathematics,
Faculty of Advanced Science and Technology, 
Kumamoto University, 2-39-1 Kurokami, 
Chuo-ku,  Kumamoto 860-8555 Japan
}
\email{andonaoya@kumamoto-u.ac.jp}

\author{Masaaki Umehara}
\address[Masaaki Umehara]{%
   Department of Mathematical and Computing Sciences,
   Tokyo Institute of Technology
   2-12-1-W8-34, O-okayama, Meguro-ku,
   Tokyo 152-8552, Japan.
}
\email{umehara@is.titech.ac.jp}

\date{May 12, 2023}

\thanks{
The first author was partially supported by 
JSPS KAKENHI Grant (C) No. 21K03228, and
the second author was partially supported by 
JSPS KAKENHI Grant  
(B) No.\  21H00981.
}

\usepackage[active]{srcltx}
\begin{document}

\maketitle
\begin{abstract}
In this paper, we prove several fundamental properties
on umbilics of a space-like or time-like surface
in the Lorentz-Minkowski space $\Lo^3$.
In particular, we show that the local behavior of
the curvature line flows of the germ of a 
space-like surface in $\Lo^3$
is essentially the same as that of a surface in Euclidean space.
As a consequence, for each positive integer $m$ 
there exists a germ of a space-like surface 
with an isolated $C^{\infty}$-umbilic (resp.~$C^1$-umbilic) of index
$(3-m)/2$ (resp. $1+m/2$).
We also show that
the indices of isolated umbilics 
of time-like surfaces in $\Lo^3$
that are not the accumulation points of
quasi-umbilics are always equal to zero.
On the other hand, when quasi-umbilics accumulate,
there exist countably many germs of time-like surfaces 
which admit an isolated umbilic with non-zero indices.
\end{abstract}

\section*{Introduction}
We denote by $\E^3$ the Euclidean 3-space
and by $\Lo^3$ the Lorentz-Minkowski 3-space of signature $(++-)$.
An immersion $f:U\to \Lo^3$ defined on a neighborhood $U$ of the origin
$o:=(0,0)\in \R^2$ is said to be {\it space-like} (resp. {\it time-like})
if its induced metric (i.e. the first fundamental form of $f$)
is Riemannian (resp. Lorentzian).
An {\it umbilic} (resp.~a {\it quasi-umbilic})
is a point, where the shape operator $A_f$ of $f$
is a scalar multiple
of the identity transformation
 (resp. has an eigen-equation with a double root 
but $A_f$ is not diagonalizable).
Quasi-umbilics never appear on a space-like surface,
but may appear on a time-like surface. We note that
the principal curvatures of a time-like surface
may not be real-valued.
The location of the umbilics on a surface $S$ in $\R^3$ does depend on the choice
of the ambient metric $\R^3=\E^3$ or $\R^3=\Lo^3$.
In fact, even when $p$ is an umbilic
of a space-like surface $S$ in $\Lo^3$,
the point $p$ may not be an umbilic point if 
we think of $S$ as lying in $\E^3$ in general (see Example \ref{ex:Ea2}).

In this paper, we focus on the study of 
curvature line flows of surfaces in $\Lo^3$, 
and so we only consider space-like surfaces and
time-like surfaces whose principal curvatures are real.
An {\it index} of an isolated umbilic on a given regular 
surface is the index of one of the curvature line flows of the 
surface at that point, which takes values in the set $\frac12\Z$ of
half-integers. 
For a curvature line flow $\mc F$ around an isolated umbilic $o$ 
in
the domain of definition of
a space-like or time-like surface in $\Lo^3$, 
we can construct another curvature line flow $\mc F^\perp$ associated with $\mc F$
satisfying $(\mc F^\perp)^\perp=\mc F$ (see Proposition~C and 
Remark \ref{rmk:629}).
When the surface is space-like, the indices of these two flows at $o$
coincide. In particular, for a given isolated umbilic on a space-like 
surface, its index is uniquely determined.
On the other hand, when the surface is time-like, 
the two indices at $o$ 
might be different (cf. 
Theorem F, see also Example \ref{eq:I2}).
 
Since the turn of the 21st century, Tari \cite{T} proved an analogue of
the Carath\'eodory conjecture for closed convex surfaces
in $\Lo^3$, and Fontenele and Xavier \cite{FX} showed 
the non-positivity
of the index of an isolated umbilic on a surface in $\E^3$
which is negatively curved except at the umbilic.
The present article, inspired by these works,  
investigates the behavior of the curvature line flows on
space-like surfaces and time-like surfaces. 

In the authors' previous work \cite{AFU} with Fujiyama,
the existence of
isolated $C^1$-differentiable umbilics with arbitrarily high indices
was shown by two ways.
One is to use the inversion sending the point at infinity
in $\E^3$ to the origin as an umbilic, and
the other is the method
using the ``Ribaucour parameter"
described in the appendix of authors' 
previous work \cite{AFU} with Fujiyama,
by which we   the pair of curvature line flows of a surface in $\E^3$
can be transformed into the pair of the eigen-flows of the Hessian 
of a certain smooth function
(see \cite{BHM} where Ribaucour's parameters is explained in terms of
Ribaucour's transformations).
We call this procedure {\it Ribaucour's reduction}, 
which is the keystone of the method of this paper.
In fact, we first show that Ribaucour's reduction can be modified
for space-like surfaces in $\Lo^3$.
Using the modified reduction, we show that 
the pair of curvature line flows around a given umbilic of a space-like surface 
in $\Lo^3$ can be realized as a pair of 
the curvature line flows
around the corresponding umbilic of a certain
regular surface in $\E^3$. By this with the result in \cite{AFU},
we can deduce the following assertion:

\medskip
\noindent
{\bf Theorem A.}
{\it For each positive integer $m$,
there exist a neighborhood $U$ of the origin $o\in \R^2$
and a space-like $C^\infty$-immersion $($resp. $C^1$-immersion$)$
$g:U\to \Lo^3$ satisfying the following properties:
\begin{enumerate}
\item $g$ is $C^\infty$-differentiable and has no umbilics on $U\setminus \{o\}$,
\item $o$ is an isolated singular point
of each of the curvature line flows of $g$
 with index $(3-m)/2$ $($resp. $1+m/2)$.
\end{enumerate}
}

Motivated by the main theorem of \cite{FX},
we show the following:

\medskip
\noindent
{\bf Proposition B.}
{\it The index of an isolated umbilic $o$ on a space-like
surface in $\Lo^3$ whose Gaussian curvature is positive}
except at $o$ is non-positive.

\medskip
We next consider time-like surfaces, and
then introduce an analogue of Ribaucour's reduction
as in the case of space-like surfaces in $\Lo^3$. 
Using it, we show the following:

\medskip
\noindent
{\bf Proposition C.}
{ \it Let $U$ be a neighborhood of the origin $o\in \R^2$
and $h:U\to \Lo^3$ a $C^\infty$-differentiable time-like surface
which has no umbilics on $U\setminus \{o\}$. 
If there exists a $C^1$-differentiable curvature line flow $\mc F$
on $U\setminus \{o\}$ whose index at $o$ is $n/2\in \frac12\Z$,
then $\mc F$ canonically induces 
another curvature line flow $\mc F^\perp$ 
$($satisfying $(\mc F^\perp)^\perp=\mc F)$ whose
index at $o$ is $-n/2$.
Moreover, if $\mc F$ is $C^{r}$-differentiable $(r\ge 1)$, 
then so is $\mc F^\perp$.
}

\medskip
When $o$ is not an accumulation point of
quasi-umbilics, the following assertion holds:

\medskip
\noindent
{\bf Proposition D.}
{\it 
Let $h:U\to \Lo^3$ be a time-like surface as in
Proposition C,
which has a curvature line flow $\mc F$
on $U\setminus \{ o\}$.
If $h$ has no quasi-umbilics  
on $U\setminus \{o\}$,
then  the index of the flow $\mc F$ at $o$
is equal to zero.
In particular, if 
the Gaussian curvature of $h$ is negative on $U\setminus \{o\}$,
then the index vanishes.}

\medskip
In this proposition, the assumption that 
$h$ has no quasi-umbilics  
on $U\setminus \{o\}$
is necessary. 
If it is dropped, 
more than two $C^1$-differentiable curvature line flows 
with non-zero indices might 
exist (cf. Example \ref{eq:I2}). 
In fact, we can show the following two assertions:

\medskip
\noindent
{\bf Theorem E.}
{\it Let $h:U\to \Lo^3$ be a
$C^\infty$-differentiable time-like surface
whose principal curvatures are
real-valued. 
If $o\in U$ is a point such that there are no umbilics on $U\setminus \{o\}$, 
then there exists a $C^1$-differentiable 
curvature line flow on $U\setminus \{o\}$ with zero index}.

\medskip
\noindent
{\bf Theorem F.}
{\it There exist countably many real analytic time-like surfaces 
with non-positive Gaussian curvature
which  admit a pair of real analytic curvature line flows 
with an isolated umbilic with indices $\pm 1$.
}

\medskip
If the curvature line flow $\mc F$ is real analytic, 
then there are no real analytic curvature line flows other than $\mc F$ and
$\mc F^\perp$ (cf. Remark \ref{rmk:1172}).
The authors do not know of any $C^r$-differentiable ($r\ge 1$)
curvature line flows on time-like surfaces 
having isolated umbilics whose indices $I$ satisfy $|I|>1$.

In Tari \cite{T},
the non-existence of umbilics on the time-like part of a convex surface
in $\Lo^3$
was pointed out (see also in Remark~\ref{rmk:903}).

\section{Preliminaries}
Let
$
\lambda(x,y)
$
be a $C^\infty$-function defined on a certain open neighborhood $U$
of the origin $o:=(0,0)$ of the $xy$-plane. 
Assuming
$
\lambda (o)=\lambda_x (o)=\lambda_y (o)=0,
$
we consider the symmetric matrix 
$$
H_\lambda(x,y):=
\pmt{\lambda_{xx}(x,y) & \lambda_{xy}(x,y) \\
\lambda_{xy}(x,y) & \lambda_{yy}(x,y)
},
$$
which is the Hessian of $\lambda$. 
We suppose that $o$ is a {\it scalar point} of
$H_\lambda$, that is,
$H_{\lambda} (0, 0)$ is a scalar multiple
of the $2\times 2$ identity matrix. 
In addition, we also assume that 
there are no scalar points of $H_{\lambda}$ on $U$ other than $o$. 
We denote by $I_{\lambda}$ the index of one of the eigen-flows 
of the matrix $H_{\lambda} (x, y)$ at $o$.
We remark that the index of the other eigen-flow is equal to $I_{\lambda}$
(see Fact \ref{fact:1128}  in the appendix).
We set
\begin{equation}\label{eq:F254}
f_\lambda(x,y):=(x,y,0)-\lambda(x,y) \nu(x,y)
+\lambda(x,y)\mb e_3,
\end{equation}
which gives a $C^\infty$-regular surface in $\E^3$ 
defined on a neighborhood of the origin $o$,
where $\mb e_3:=(0,0,1)$ and
\begin{equation}\label{eq:N181}
\nu:=\frac{1}{1+\lambda_x^2+\lambda_y^2}
\left(2\lambda_x,2\lambda_y,\lambda^2_x+\lambda^2_y-1\right)
\end{equation}
gives a unit normal vector field of $f_\lambda$.
In this setting, $(x,y)$ is called
{\it Ribaucour's parametrization} of $f_\lambda$.
The following fact is classically known (cf. \cite[Appendix A]{AFU}):

\begin{Fact}\label{fact:192}
The origin of $\E^3$
is an isolated umbilic of $f_{\lambda}$ 
whose index coincides with $I_{\lambda}$.
\end{Fact}

\begin{Remark}
One can show that
 $p\in U$ is an umbilic of the surface $f_\lambda$ if and only if
$(\lambda_{xx}-\lambda_{yy},\lambda_{xy})$ vanishes at $p$
(see the proof of \cite[Fact A.1]{AFU}). 
\end{Remark}

Here, we give elementary examples:

\begin{Exa}
We consider an ellipsoid
\begin{equation}\label{eq:ellisopid}
\frac{x^2}{a^2}+y^2+\frac{z^2}{b^2}=1\qquad (1<a\le b).
\end{equation}
When $a=b$, the ellipsoid has two umbilics of index $1$,
and when $a<b$, it has four umbilics of 
index $1/2$ (cf. \cite[Section 16]{UY}).
\end{Exa}

\begin{Exa}
We consider the function
$\lambda(x,y):=x^4-y^4$.
Then $H_\lambda$ is diagonal, and
$o$ is an isolated scalar point of $H_\lambda$.
So, the associated regular surface $f_\lambda$
(cf. \eqref{eq:F254})
has an isolated umbilic at $o$, whose index is equal to zero.
\end{Exa}

\begin{Exa}\label{exa:Z402}
We next consider the polynomial
\begin{equation}\label{eq:lambda_n}
\lambda(x,y):=\op{Re}(\zeta^n)\qquad (\zeta:=x+iy,\,\, n\ge 3).
\end{equation}
Since
$
(\lambda)_{\zeta\zeta}=n(n-1)\zeta^{n-2}/2
$,
the winding number 
of the complex-valued function $\lambda_{\zeta\zeta}$ 
with respect to 
a counterclockwise circle of small radius centered 
at the origin $o$
is equal to $n-2$, and
the associated regular surface $f_{\lambda}$ has an isolated 
umbilic at $o$ with index $1-n/2$ (cf. \cite[(B-1)]{AFU}).
\end{Exa}

In these examples, the following well-known fact can be observed.

\begin{Fact}\label{prop:294}
 For each positive integer $m$,
there exists a $C^\infty$-regular surface in $\E^3$ having an
isolated umbilic whose index is equal to $(3-m)/2$.
\end{Fact}

\section{Umbilics of space-like surfaces in $\Lo^3$}

In this section, ``space-like surfaces"  always mean 
space-like regular surfaces.
As in the case of regular surfaces in $\E^3$,
space-like surfaces in $\Lo^3$ has exactly two
distinct principal directions at each non-umbilic point.
So, around an isolated umbilic,
a pair of curvature line flows is induced.
In this section, we investigate them.
More precisely, 
we modify Ribaucour's reduction given in 
\cite[Appendix A]{AFU} for space-like surfaces in $\Lo^3$ and 
prove Theorem A. 
We first give an example of a
space-like surface with umbilics:

\begin{Exa}\label{ex:Ea2}
We let $E_a$ be the ellipsoid given in \eqref{eq:ellisopid} 
by setting $a=b$. If $a>1$, then umbilics of $E_a$ as a 
surface in $\E^3$ are the two points $(0,\pm1,0)$.
However, if we think of $E_a$ as lying in $\Lo^3$, then 
these points $(0,\pm1,0)$ lie on the time-like part of $E_a$ 
and cannot be umbilics in $\Lo^3$ (cf. Remark \ref{rmk:903}). 
In fact, by a straightforward calculation,
one can check that the space-like part of $E_a$ has exactly
four space-like umbilics
$\frac{1}{\sqrt{2}}(\pm \sqrt{a^2-1},0,\pm \sqrt{a^2+1})$ in
$\Lo^3$.
\end{Exa}

We let $g:U\to \Lo^3$ be a space-like immersion
defined on a neighborhood $U$ of the origin
$o$ in the $xy$-plane.
Since we are interested in local properties of surfaces, 
we may set
$$
g(x,y):=(x,y,\phi(x,y)),
$$
where $\phi$ is a $C^\infty$-function satisfying
$
\phi(o)=\phi_x(o)=\phi_y(o)=0.
$
We denote by
$\lq\lq \cdot "$ 
the
Lorentzian inner product on $\Lo^3$, 
and consider the point 
$$
Q(x,y):=(\xi(x,y),\eta(x,y),0)
$$
in the $xy$-plane satisfying
\begin{equation}\label{eq:Q}
Q+\mu \mb e_3=g+\mu \nu\qquad (\mb e_3:=(0,0,1)),
\end{equation}
where $\mu$ is a certain $C^\infty$-function on $U$
and
$$
\nu(x,y):=\frac{-1}{\delta_+(x,y)}(\phi_x(x,y),\phi_y(x,y),1)\qquad
(\delta_+:=\sqrt{1-\phi_x^2-\phi_y^2}\,)
$$
gives a unit normal vector field of $g$,
that is, $|\nu\cdot \nu|=1$.
Comparing the third components of the
both sides of \eqref{eq:Q},
we obtain the relation
$$
\mu= \frac{\phi \delta_+}{1+\delta_+}.
$$
Since $(0,0)$ is a critical point of the function 
$\phi$, we have
$\mu(0,0)=0$ and $d\mu(0,0)=0$
and so $dg=dQ$ holds at $(0,0)$. In particular, 
$(\xi,\eta)$ can be taken as a new local coordinate system centered at $o$.
Differentiating \eqref{eq:Q} by $\xi$ and $\eta$ respectively,
and taking the Lorentzian inner products of 
the both sides of them with $\nu$, we have
$$
Q_\xi\cdot\nu-\mu_\xi^{} \nu_3=-\mu_\xi^{}, \qquad
Q_\eta\cdot\nu-\mu_\eta^{} \nu_3=-\mu_\eta^{},
$$
where $\nu=(\nu_1,\nu_2,\nu_3)$.
Since $Q_\xi=(1,0,0)$ and $Q_\eta=(0,1,0)$,
we obtain the following:
$$
\mu_\xi^{}=\frac{-\nu_1}{1-\nu_3},\qquad
\mu_\eta^{}=\frac{-\nu_2}{1-\nu_3},
$$
which correspond to the stereographic projection
of the hyperbolic space 
$$
\{(x,y,z)\in \Lo^3\,;\, x^2+y^2-z^2=-1,\,\, z<0\}
$$
to the $xy$-plane.
By this, as an analogue of 
\eqref{eq:N181}, we obtain
\begin{equation}\label{eq:n524}
\nu(\xi,\eta)=\frac{1}{\mu_\xi^2+\mu_\eta^2-1}
(2\mu_\xi^{},2\mu_\eta^{},\mu_\xi^2+\mu_\eta^2+1).
\end{equation}
So we have that
\begin{equation}\label{eq:g529}
g(\xi,\eta)=(\xi,\eta,0)-\mu(\xi,\eta) 
\nu(\xi,\eta)+\mu(\xi,\eta)\mb e_3,
\end{equation}
which is an analogue of \eqref{eq:F254}.
We call the above 
procedure {\it space-like Ribaucour's reduction}
and the coordinate system $(\xi,\eta)$ {\it space-like Ribaucour's 
parametrization} of $g$.
Let $\gamma(t)$ be a regular curve in the $\xi\eta$-plane.
This curve is an orbit of one of the curvature line flows of $g$ if and only if
$d(\nu\circ \gamma)(t)/dt$ and $d(g\circ\gamma)(t)/dt$ are linearly dependent.
By the same argument
as in \cite[Appendix A]{AFU},
the equation $\det(\nu,dg,d\nu)=0$
characterizes the curvature line flows of $g$,
and \eqref{eq:g529} yields that
\begin{equation}\label{eq:535}
\det(\nu,dg,d\nu)=
\det\pmt{
\nu_1 & d\xi       & d\nu_1 \\
\nu_2 & d\eta      & d\nu_2 \\
\nu_3 & \mu_\xi^{} d\xi+\mu_\eta^{} d\eta   & d\nu_3
},
\end{equation}
where $``\det"$ means the determinant function.
Since $g$ is space-like, we have
$
\nu_1^2+\nu_2^2-\nu_3^2=-1
$
and
\begin{equation}\label{eq:554}
\mu_\xi^{} \nu_1+\mu_\eta^{} \nu_2
={1+\nu_3}.
\end{equation}
By \eqref{eq:n524}, it holds that
\begin{equation}\label{eq:441}
d\nu=-\frac{dk}{k}\nu+\frac{2}{k}(d\mu_\xi^{},d\mu_\eta^{}, 
\mu_\xi^{} d\mu_\xi^{}+\mu_\eta^{} d\mu_\eta),
\end{equation}
where $k:=\mu_\xi^2+\mu_\eta^2-1$.
By  
\eqref{eq:535}, \eqref{eq:554} 
 and \eqref{eq:441}, we have (see also \cite[Appendix A]{AFU})
\begin{align}\label{eq:581}
\det(\nu,dg,d\nu)&=
\frac{2}{k}\det\pmt{
\nu_1 & d\xi       & d\mu_\xi^{} \\
\nu_2 & d\eta      & d\mu_\eta^{} \\
\mu_\xi^{} \nu_1+\mu_\eta^{} \nu_2-1 & 
\mu_\xi d\xi+\mu_\eta d\eta   & \mu_\xi^{} d\mu_\xi^{}+\mu_\eta^{} 
d\mu_\eta^{}
}  \\ \nonumber
&=
\frac{2}{k}
\det\pmt{
\nu_1 & d\xi       & d\mu_\xi^{} \\
\nu_2 & d\eta      & d\mu_\eta^{} \\
-1 & 0   & 0
} 
=\frac{-2}{k}(d\xi, d\eta) S_\mu
\pmt{d\xi \\ d\eta},
\end{align}
where
$$
S_\mu:=
\pmt{\mu_{\xi\eta}^{} & (\mu_{\eta\eta}^{}-\mu_{\xi\xi}^{})/2 \\
(\mu_{\eta\eta}^{}-\mu_{\xi\xi}^{})/2 & -\mu_{\xi\eta}^{}
}.
$$
Thus, the curvature line flows of $g$ just coincide with the
null direction flows of $S_\mu$ 
(see the appendix in this paper for the definition of null directions). 

\begin{Rmk}\label{rmk:629}
If $\mb v:=u(\partial/\partial \xi)_p+v(\partial/\partial \eta)_p$ is a
tangent vector at $p\in U$
giving a null direction of $S_\mu$, then
$\mb v^\perp:=
-v(\partial/\partial \xi)_p+u(\partial/\partial \eta)_p$ 
is the $90^\circ$-rotation of $\mb v$ 
giving also a null direction of $S_\mu$.
So if $\mc F$ is a curvature line flow of the space-like surface
$g$, then the $90^\circ$-rotation $\mc F^\perp$ in the $\xi\eta$-plane
also gives a curvature line flow of $g$.
This fact is one of the strengths of 
Ribaucour's parametrizations.
As a consequence, the curvature line flows of $g$ can be considered as a pair 
$(\mathcal{F} , \mathcal{F}^{\perp})$.
\end{Rmk}

The characteristic vector field $\mb v_{S_\mu}^{}$
 (cf. the appendix) of $S_\mu$ is given by
$$
\mb v_{S_\mu}^{}=2\mu_{\xi\eta}^{}\frac{\partial}{\partial \xi}+
(\mu_{\eta\eta}^{}-\mu_{\xi\xi}^{}) \frac{\partial}{\partial \eta}.
$$
The $90^\circ$-rotation of this vector field is
$
(\mu_{\xi\xi}^{}-\mu_{\eta\eta}^{}) \frac{\partial}
{\partial \xi}
+
2\mu_{\xi\eta}^{}\frac{\partial}{\partial \eta},
$
which coincides with the characteristic vector field of 
the Hessian $H_\mu$ of the function $\mu$.
By Proposition \ref{prop:989} in the appendix,
the null direction flows of $S_\mu$ 
are obtained by the $45^\circ$-rotation 
of the eigen-flows of $S_\mu$ in the $\xi\eta$-plane.
Thus, the null direction flows of $S_\mu$ 
can be identified with the eigen-flows of $H_\mu$.
Using these discussions,
we obtain the following:

\begin{Theorem}\label{prop:505}
For a given $C^\infty$-function $\mu$ on a neighborhood $U$
of the origin $o$ in the $\xi\eta$-plane, the map
$g:U\to \Lo^3$ given by \eqref{eq:g529} with \eqref{eq:n524}
is a space-like immersion.
Any congruence class of germs of space-like immersions in $\Lo^3$
is obtained in this manner.
Moreover, $p\in U$ is an umbilic if and only if
$H_\mu$ is a scalar matrix at $p$. 
If $p\in U$ is not an umbilic,
the principal directions of $g$ at $p$
can be obtained by the eigen-directions of $H_\mu$
as in \eqref{eq:Tm}.
In particular, if $o$ is an isolated umbilic of $g$,
then the pair of curvature line flows of $g$ exists
around $o$
and they have the same index at $o$.
\end{Theorem}

\begin{proof}
It is sufficient to show that
$p$ is an umbilic point of $h$ when $S_\mu$ vanishes at $p$,
which follows from \eqref{eq:581}, since
$\det(\nu,dg,d\nu)$ vanishes at $p$ if and only if 
$p$ is an umbilic.
\end{proof}

Using Theorem \ref{prop:505}, 
we prove Theorem A in the introduction:

\begin{proof}[Proof of Theorem A]
We fix a positive integer $m$.
By
Fact~\ref{prop:294},
there exists a regular surface $f$ in $\E^3$
which has an isolated
umbilic whose index is equal to $(3-m)/2$.
Then there exist a new local coordinate system
$(x,y)$ and
a $C^\infty$-function $\lambda(x,y)$
such that $f$
is expressed as
\eqref{eq:F254}
with \eqref{eq:N181}.
The pair of curvature line flows of $f$ coincides with the pair of
 eigen-flows of $H_{\lambda}$, and the isolated umbilic of 
$f$ corresponds to an isolated scalar point of $H_{\lambda}$.

By setting $\mu:=\lambda$, we
define a regular space-like surface $g$ in $\Lo^3$
given by \eqref{eq:g529} with \eqref{eq:n524}. By Theorem~\ref{prop:505}
(see also \eqref{eq:581}), 
one of the two curvature line flows of $g$ coincides with
either of the eigen-flows of $H_\lambda$,
and the isolated 
scalar point of $H_{\lambda}$ corresponds to an isolated umbilic of $g$.
So the index of the isolated umbilic of $g$ is $(3-m)/2$.

We next consider the
$C^1$-differentiable function
$$
\lambda_m(\xi,\eta)=|z|^2 \tanh\left(|z|^{-a}\op{Re}(z^m/|z|^m)\right)
\qquad (z:=\xi+\sqrt{-1}\eta,\,\,0<a<1)
$$
given in \cite{AFU}.
By setting $\xi=\rho \cos t$ and $\eta=\rho \sin t$ ($\rho>0$),
the function $\lambda_m$ induces a function
$$
\tilde \lambda_m(\rho,t):=\rho^2 \tanh(\rho^{-a}\cos m t),
$$
of variables $\rho,t$.
In \cite[Section 6]{AFU},
it was proved that
the indices of the eigen-flows of $H_{\lambda_m}$ at $o$
are equal to $1+m/2$. By setting
$\mu :=\lambda_m$,
the immersion $g_m$ defined by 
\eqref{eq:g529} with \eqref{eq:n524}
is $C^1$-differentiable at $o$, and
$C^\infty$-differentiable on $V\setminus \{o\}$
for a sufficiently small neighborhood $V$ of $o$.
The indices of the curvature line flows of
$g_m$ at $o$ are equal to $1+m/2$. Thus, Theorem A is obtained.
\end{proof}

It is well-known that the Gaussian curvature of a 
(space-like or time-like) 
surface in $\Lo^3$ has
the opposite sign of that of the same surface in $\E^3$ (thinking of
$\E^3$ as the space $\R^3$ with the canonical Euclidean metric).
Regarding this,
we now prove Proposition B:

\begin{proof}[Proof of Proposition B]
Let $g(\xi,\eta)$ be the space-like immersion
given by 
\eqref{eq:n524}
and
\eqref{eq:g529} using a $C^\infty$-function $\mu$
defined on a neighborhood $U$
of the origin of the $\xi\eta$-plane.
We let $I\!I_g:=L_Sd\xi^2+2M_S d\xi d\eta+N_Sd\eta^2$
be the second fundamental form of $g$.
We may assume that $(0,0)$ corresponds to the umbilic $o$.
From now on, we will show the following
equivalency;
$$
\text{$I\!I_g$ is negative definite}
\Longleftrightarrow 
\text{$H_\lambda$ is negative definite}
\Longleftrightarrow 
\text{$I\!I_{f_\lambda}$ is negative definite},
$$
where $I\!I_{f_\lambda}$ is the second fundamental form 
of the surface $f_\lambda$ induced by $\lambda:=\mu$ in
the Euclidean $3$-space.
Then, by applying the theorem in \cite{FX},
we can conclude that
the indices of the curvature line flows of $f_\lambda$ at $o$
are non-positive, and so, the space-like surface $g$ at $o$
has the same property.
By a straightforward computation,
we have 
\begin{align*}
L_S =\dfrac{2\mu_{\xi \xi}}{q_S^{}}& 
+\dfrac{4\mu (\mu^2_{\xi \eta} +\mu^2_{\xi \xi})}{q^2_S}, \qquad
M_S =\dfrac{2\mu_{\xi \eta}}{q_S^{}} 
+\dfrac{4\mu \mu_{\xi \eta} (\mu_{\xi \xi} +\mu_{\eta \eta})}{q^2_S}, \\
&\phantom{aaaaqqqqq} N_S =\dfrac{2\mu_{\eta \eta}}{q_S^{}} 
+\dfrac{4\mu (\mu^2_{\xi \eta} +\mu^2_{\eta \eta})}{q^2_S}, 
\end{align*}
where $q_S^{}:=1-\mu_\xi^2-\mu_\eta^2$.
Hence
\begin{equation}\label{eq:A}
L_SN_S-M_S^2
=\frac{4 \left(\mu_{\xi\xi}^{} \mu_{\eta\eta}^{}-\mu_{\xi\eta}^2\right) 
D_S}{q_S^4},
\end{equation}
where
$$
D_S:=q_S^2
+2\mu(\mu_{\xi\xi}^{}+\mu_{\eta\eta}^{})q_S^{}+4\mu^2
 (\mu_{\xi\xi}^{} \mu_{\eta\eta}^{}-\mu_{\xi\eta}^2).
$$
Since 
\begin{equation}\label{eq:726}
\mu(o)=\mu_\xi(o)=\mu_\eta(o)=0,
\end{equation}
we have  $q_S^{}>0$ and $D_S>0$
at $(\xi,\eta)=o$.
So there exists a neighborhood $V(\subset U)$ of $o$
such that $q_S^{}$ and $D_S$ are positive on $V$.
Then (2.8) implies that
$L_S N_S -M^2_S <0$ is equivalent to 
\begin{equation}\label{eq:733}
\mu_{\xi\xi}^{} \mu_{\eta\eta}^{}-\mu_{\xi\eta}^2<0
\end{equation}
on $V\setminus \{o\}$.
We then set $\lambda:=\mu$, and 
let $f_\lambda(\xi,\eta)$ be the regular surface 
given by \eqref{eq:F254}
and \eqref{eq:N181}.
The second fundamental form $I\!I_{f_\lambda}$
of the surface $f_\lambda$ in $\E^3$ can be written as
$I\!I_{f_\lambda}=L_Ed\xi^2+2M_Ed\xi d\eta+N_Ed\eta^2$.
Again, by a straightforward computation,
we have 
\begin{equation}\label{eq:B}
L_EN_E-M_E^2=\frac{4 \left(\lambda_{\xi\xi} \lambda_{\eta\eta}-\lambda_{\xi\eta}^2\right) 
D_E}{q_E^4},
\end{equation}
where $q_E^{}:=1+\lambda_\xi^2+\lambda_\eta^2$ and
$$
D_E:=q_E^2
-2\lambda(\lambda_{\xi\xi}^{}+\lambda_{\eta\eta}^{})q_E^{}+4\lambda^2 
(\lambda_{\xi\xi}^{} \lambda_{\eta\eta}^{}-\lambda_{\xi\eta}^2).
$$
Since $\lambda=\mu$, the inequality \eqref{eq:733}
is equivalent to $L_EN_E-M_E^2<0$
(that is, the Gaussian curvature of $f_\lambda$ is negative) 
on $W\setminus \{ o\}$ for a sufficiently small neighborhood $W$ of $o$.
Thus, by the theorem in \cite{FX},
the indices of the curvature line flows of $f_\lambda$ at $o$
are non-positive.
So, we obtain the conclusion.
\end{proof}

\section{Umbilics of time-like surfaces in $\Lo^3$}

In this section, ``time-like surfaces" 
always mean time-like regular surfaces.
For time-like surfaces in $\Lo^3$,
even at non-umbilic points, 
principal directions might not exist in general,
and even if the directions exist, 
the two principal directions might coincide
(such cases happen when they coincide with
a null direction of the first fundamental form).
In this section, as in the case of space-like surfaces,
we give an analogue of Ribaucour's parametrization
for time-like surfaces and prove Proposition C.
We let
$h:U\to \Lo^3$ be a time-like immersion
defined on a neighborhood $U$ of the origin
$o$ in the $yz$-plane.
We may set 
$$
h(y,z):=(\psi(y,z),y,z),
$$
where $\psi$ is a certain $C^\infty$-function defined on $U$ satisfying 
$
\psi(0,0)=\psi_y(0,0)=\psi_z(0,0)=0.
$
We consider the point 
$
Q(y,z):=(0,\xi(y,z),\eta(y,z))
$
satisfying
\begin{equation}\label{eq:Q2}
Q+\mu \mb e_1=h+\mu \nu \qquad \big(\mb e_1:=(1,0,0)\big),
\end{equation}
where $\mu$ is a certain $C^\infty$-function on $U$.
Moreover,
$$
\nu(y,z):=\frac{-1}{\delta_-(y,z)}
(1,-\psi_y(y,z),\psi_z(y,z))\qquad
(\delta_-:=\sqrt{1+\psi_y^2-\psi_z^2}\,)
$$
gives a unit normal vector field of $h$.
Comparing the first components of the both sides of \eqref{eq:Q2},
we have
$$
\mu= \frac{\psi \delta_-}{1+\delta_-}.
$$
Since $(0,0)$ is a critical point of the function $\psi$, we have
$\mu(0,0)=0$ and $d\mu(0,0)=0$.
In particular, $dh=dQ$ holds at $(0,0)$, and 
$(\xi,\eta)$ can be taken as a new local coordinate system centered at $o$.
Differentiating \eqref{eq:Q2} by $\xi$ and $\eta$ respectively,
and taking the Lorentzian inner products of both sides of them
with $\nu$, we have
$$
Q_\xi\cdot\nu+\mu_\xi^{} \nu_1=\mu_\xi^{},\qquad
Q_\eta\cdot\nu+\mu_\eta^{} \nu_1=\mu_\eta^{}.
$$
Since $Q_\xi=(0,1,0)$ and $Q_\eta=(0,0,1)$, we have
$$
\mu_\xi^{}=\frac{\nu_2}{1-\nu_1},\qquad
\mu_\eta^{}=\frac{-\nu_3}{1-\nu_1},
$$
which correspond to the stereographic projection of the subset 
$\{(x,y,z)\in \Lo^3\,;\, x^2+y^2-z^2=1\}\setminus\{x=1\}$
of de Sitter plane to the $yz$-plane.
So we have
\begin{equation}\label{eq:nu729}
\nu(\xi,\eta)=\frac{1}{-1-\mu_\xi^2+\mu_\eta^2}(1-\mu_\xi^2+\mu_\eta^2,
-2\mu_\xi^{},2\mu_\eta^{})
\end{equation}
and
\begin{equation}\label{eq:h734}
h(\xi,\eta)=(0,\xi,\eta)-\mu(\xi,\eta) \nu(\xi,\eta)+\mu(\xi,\eta)\mb e_1
\end{equation}
analogous to \eqref{eq:F254} and \eqref{eq:g529}.
We call the procedure {\it time-like Ribaucour's reduction}
and $(\xi,\eta)$ {\it time-like Ribaucour's parametrization} of $h$.
Since $h$ is a time-like surface,
$
\nu_1^2+\nu_2^2-\nu_3^2=1
$
holds. So we have $\mu_\xi^{} \nu_2+\mu_\eta^{} \nu_3={1+\nu_1}$.
By \eqref{eq:nu729},
we have
$$
d\nu=-\frac{dk}{k}\nu+\frac{2}{k}
(-\mu_\xi^{} d\mu_\xi^{}+\mu_\eta^{} d\mu_\eta^{}, 
-d\mu_\xi^{},d\mu_\eta^{}),
$$
where $k:=-1-\mu_\xi^2+\mu_\eta^2$.
As in the case of space-like surfaces, 
we have
(cf. \eqref{eq:581})
\begin{equation}\label{eq:930}
\det(\nu,dh,d\nu)=
\det\pmt{
\nu_1 & \mu_\xi^{} d\xi+\mu_\eta^{} d\eta   & d\nu_1\\
\nu_2 & d\xi       & d\nu_2 \\
\nu_3 & d\eta      & d\nu_3 
} 
=
\frac{-2}{k}
(d\xi,d\eta)T_\mu \pmt{d\xi\\ d\eta},
\end{equation}
where
\begin{equation}\label{eq:Tm}
T_\mu:=
\pmt{\mu_{\xi\eta}^{} & (\mu_{\eta\eta}^{}+\mu_{\xi\xi}^{})/2 \\
(\mu_{\eta\eta}^{}+\mu_{\xi\xi}^{})/2 & \mu_{\xi\eta}^{}
}.
\end{equation}
Thus, if a curvature line flow of $h$ on $U\setminus \{o\}$
exists, then it corresponds to a 
null direction flow (see the appendix) of the symmetric matrix $T_\mu$.

\begin{Remark}\label{rmk:TcT}
We consider two matrices
$$
T=\pmt{c & (a+b)/2 \\
      (a+b)/2 & c},\quad
\check T:=E_2\pmt{a & c \\
      c & b}=
\pmt{a & c \\
      -c & -b}
 \quad (E_2:=\pmt{1 & 0 \\ 0 & -1}).
$$
Then the null vectors of $T$ coincide with the eigenvectors
of $\check T$. So the null vectors of $T_\mu$
coincide with the eigenvectors of the matrix $\check T_\mu:=E_2H_\mu$,
where $H_\mu$ is the Hessian matrix of the function $\mu$.
\end{Remark}

Using the above observation, we can show the following:

\begin{Theorem}\label{prop:758}
For a given $C^\infty$-function $\mu$ on a neighborhood $U$
of the origin in the $\xi\eta$-plane, the map
$h:U\to \Lo^3$ 
given by \eqref{eq:h734} with
\eqref{eq:nu729}
is a time-like immersion.
Any congruence class of
time-like immersions in $\Lo^3$
is  obtained in this manner.
Moreover, $p\in U$ is an umbilic $($resp.\!~an umbilic 
or a quasi-umbilic$)$
 if and only if
$T_\mu$ $($resp. $\det T_\mu)$
vanishes at $p$. If
$p$ is not an umbilic, 
the principal directions of $h$ at $p$
can be characterized by the null vectors
 of $T_\mu$
of \eqref{eq:Tm}.
\end{Theorem}

\begin{proof}
It is sufficient to show the assertions on umbilics and quasi-umbilics.
We first consider umbilics of $h$.
It is sufficient to show that
$p$ is an umbilic point of $h$ when $T_\mu$ 
vanishes at $p$,
which follows by the same reason
as the proof of Theorem~\ref{prop:505}.

We next consider quasi-umbilics of $h$.
By
\eqref{eq:930},
$p\in U$ is a quasi-umbilc precisely when $T_{\mu}$
has a unique null-dierction, that is, when $\det T_\mu=0$ but $T_\mu\ne 0$
at $p$. So we obtain the conclusion.
\end{proof}

As an application of
Theorem \ref{prop:758}, we prove Proposition C in the introduction.

\begin{proof}[Proof of Proposition C]
By Theorem \ref{prop:758},
we may assume that the time-like immersion $h:U\to \Lo^3$
is given by \eqref{eq:h734} with \eqref{eq:nu729} 
associated with 
a $C^\infty$-function $\mu$ defined on 
a neighborhood $U$ of the origin in the $\xi\eta$-plane.
We let $\mc F$ be a $C^1$-curvature line flow of $h$
on $U\setminus \{o\}$.
Then $\mc F$ is generated locally by a null vector field
$$
(\mb v(\xi,\eta):=)u(\xi,\eta)\frac{\partial}{\partial \xi}+v(\xi,\eta)\frac{\partial}{\partial \eta}
$$
of $T_\mu$.
For simplicity, we set
\begin{equation}\label{eq:simple}
2T_\mu(\xi,\eta)=\pmt{a(\xi,\eta) & b(\xi,\eta) \\ b(\xi,\eta) & a(\xi,\eta)} \qquad
(a:=2\mu_{\xi\eta}^{},\,\, b:=\mu_{\eta\eta}^{}+\mu_{\xi\xi}^{}).
\end{equation}
In this situation, $\mb v(\xi,\eta)$
is a null vector of $T_\mu(\xi,\eta)$ if
and only if
\begin{equation}\label{eq:1083}
0=2(u,v) T_\mu \pmt{u \\ v}=(u,v)\pmt{au+bv \\ bu+av}
=a(u^2+v^2)+2b uv.
\end{equation}
Since the right hand side of this equation is invariant when swapping
$u$ and $v$, the vector field 
$$
\mb v^\perp(\xi,\eta):=v(\xi,\eta)\frac{\partial}{\partial \xi}
+u(\xi,\eta)\frac{\partial}{\partial \eta} 
$$
also yields a principal direction of $h$ at each point $(\xi,\eta)\in U$.
Then $\mb v^\perp$ generates another curvature line flow $\mc F^\perp$.
By definition, $(\mc F^\perp)^\perp=\mc F$ and
$\mc F^\perp$ is $C^r$-differentiable  ($r\ge 1$)  if and only if
$\mc F$ is $C^r$-differentiable.
Since the transformation $(x,y)\mapsto (y,x)$ in $\R^2$ is
orientation reversing,
the index of $\mc F$ at $o$ is $n/2$ $(n\in \Z)$ if and only if
the index of $\mc F^\perp$ 
at $o$ is $-n/2$, proving the assertion.
\end{proof}

\begin{Remark}\label{rmk:903}
If we denote by
$$
E_Td\xi^2+2F_Td\xi d\eta+G_T{d\eta^2},\qquad
L_Td\xi^2+2M_Td\xi d\eta+N_T{d\eta^2}
$$
the first and the second fundamental forms of a 
time-like surface $h$, then as an analogue to
\eqref{eq:A} and \eqref{eq:B}, 
one can show that
\begin{equation}\label{eq:1075}
L_TN_T-M_T^2=\frac{4 \left(\mu_{\xi\xi}^{} \mu_{\eta\eta}^{}
-\mu_{\xi\eta}^2\right) 
D_T}{q_T^4},
\end{equation}
where $q_T^{}:=1+\mu_\xi^2-\mu_\eta^2$ and
$$
D_T:=q_T^2
-2\mu(\mu_{\xi\xi}^{}-\mu_{\eta\eta}^{})q_T^{}
-4\mu^2 \left(\mu_{\xi\xi}^{} \mu_{\eta\eta}^{}
-\mu_{\xi\eta}^2\right).
$$
If $h$ is strictly locally convex, 
then $\mu_{\xi\xi}^{} \mu_{\eta\eta}^{}
-\mu_{\xi\eta}^2$ is positive at $(\xi,\eta)=(0,0)$.
In this case, the Gaussian curvature of $h$ is negative.
Since $E_T=-G_T=1$ and $F_T=0$ at $(\xi,\eta)=(0,0)$,
the sign of the determinant of the shape operator
is negative near $(0,0)$. So,
$h$ does not admit any umbilics on a neighborhood of $(0,0)$, 
which confirms the 
observation of Tari \cite{T} mentioned in the last paragraph
of the introduction.
\end{Remark}

\begin{proof}[Proof of Proposition D]
For simplicity, we write $T_\mu$
as in \eqref{eq:simple}.
Since $h(\xi,\eta)$ has no quasi-umbilics on $U\setminus \{o\}$,
there are two distinct principal directions at each 
$(\xi,\eta)\in U\setminus \{o\}$.
By Theorem~\ref{prop:758},
$T_\mu$ has two null directions at $(\xi,\eta)$. Hence
$$
0>4\det T_\mu(\xi,\eta)=a(\xi,\eta)^2-b(\xi,\eta)^2
\qquad ((\xi,\eta)\in U\setminus \{o\}).
$$
We let $P(T_pU)$ be the projective space associated to
the tangent space at each point $p$ of
the open submanifold $U$ of $\R^2$.
We denote by 
\begin{equation}\label{eq:1038}
T_pU\ni \mb v\mapsto [\mb v]\in P(T_pU)
\end{equation}
the corresponding canonical projection.
We set
$
\mb v=\pm\Big(\cos \theta(\xi,\eta),\sin \theta(\xi,\eta)\Big)^T,
$
and consider a null direction field $[\mb v]$ 
of $T_\mu$
defined on $U\setminus \{o\}$, where 
the superscript $``T"$ denotes the transpose of the matrix.
Then, the identity
\begin{equation}\label{eq:angle800}
0=\mb v^T \,\pmt{a & b \\ b & a} \mb v=a+b\sin 2\theta
\end{equation}
holds on $U\setminus \{o\}$.
Since $|b|>|a|$, the function $\sin 2\theta$ never attains the
values $\pm 1$
on $U\setminus \{o\}$. So we conclude that
the index of the flow at $o$ induced by $[\mb v]$ is
equal to zero, proving the first assertion.

We next consider the case that
$h$ is negatively curved on $U\setminus \{o\}$.
Since $h$ is time-like, the shape operator of $h$ has 
two eigenvalues with different signs
at each point of $U\setminus \{o\}$
(in this setting, $\mu_{\xi\xi}\mu_{\eta\eta}-\mu_{\xi\eta}^2$ is
positive on a sufficiently small neighborhood of $o$).
Thus, there are no quasi-umbilics on $U\setminus \{o\}$,
and we obtain the second assertion.
\end{proof}

\begin{Exa}\label{eq:2424}
If we set
$
\mu(\xi,\eta):=\xi^2+\xi^4-\eta^2+\eta^4,
$
then the time-like surface $h(\xi,\eta)$ 
associated with $\mu$ (cf. \eqref{eq:h734}) 
has an isolated umbilic at $o:=(0,0)$.
Since $T_{\mu}$ has two  distinct
null directionsaway from $o$, 
the surface $h$ has no quasi-umbilics. So, 
by Proposition D,
the index at $o$ is equal to zero.
Since 
$$
\mu_{\xi\xi}\mu_{\eta\eta}-\mu^2_{\xi\eta}=4 \left(6 \xi^2+1\right) \left(6 \eta^2-1\right)
<0,
$$
\eqref{eq:1075} implies that
the Gaussian curvature of this surface near the origin is positive.
\end{Exa}

We next prove Theorem E: 

\begin{proof}[Proof of Theorem E]
We set 
\begin{equation}\label{eq:1124}
a:=2\mu_{\xi\eta},\qquad b:=\mu_{\xi\xi}+\mu_{\eta\eta}.
\end{equation}
Since $h$ has real principal curvatures at each point,
the function $b^2-a^2$ is non-negative on $U$, and so
there exists a continuous function $\phi$ such that
$
\phi^2=b^2-a^2.
$
We set 
$$
\mb v_1:=(-b+\phi)\frac{\partial}{\partial \xi}+a\frac{\partial}{\partial \eta},
\qquad \mb v_2:=-a\frac{\partial}{\partial \xi}+(b+\phi)\frac{\partial}{\partial \eta}.
$$
Then these two vector fields yield null vector fields of $T_\mu$ 
in the $\xi\eta$-plane.
Since
$$
\det(\mb v_1,\mb v_2)
=\det\pmt{-b+\phi & -a \\ a & b+\phi}
=(b^2-a^2)-b^2+a^2=0,
$$
$\mb v_1$ and $\mb v_2$ are
linearly dependent at each point on $U$.
Moreover, ${\bf v}_1$ and ${\bf v}_2$  do not have common zeros on $U\setminus \{o\}$. 
Assuming for a contradiction, $\mb v_1=\mb v_2=\mb 0$ at some point $p\in U\setminus \{o\}$. 
Then we deduce that $a=b=0$ at $p$.  Then, by Theorem~\ref{prop:758},  
$p$ is an umbilic, contradicting 
our assumption that  there are no umbilics on $U\setminus \{o\}$.
Since $\mb v_1$ and $\mb v_2$ are continuous vector fields,
they generate a $C^1$-curvature
line flow $\mc F_0$ defined on $U\setminus \{o\}$.

As one of the possibilities of $\phi$, 
we may set  $\phi:=\op{sgn}(b)\sqrt{b^2-a^2}$. Since
\begin{equation}\label{eq:1161}
(b^2-a^2+b\phi)(b^2-a^2-b\phi)
=-(b^2-a^2)a^2\le 0,
\end{equation}
we have $b^2-a^2-b\phi\le 0$ and so
$$
(\phi-b)^2-a^2=2(b^2-a^2-b\phi) \le 0.
$$
Then we have
$$
a^2 -(b+\varphi )^2 =-2(b^2 -a^2+\varphi b)\leq 0.
$$ 
So  if we consider the $\xi\eta$-plane as the Lorentz-Minkowski space
of signature $(+-)$,
the vector fields
$\mb v_1$ and $\mb v_2$ are not space-like.
Consequently 
the indices of the flows $\mc F_0$ and $\mc F^\perp_0$
are equal to zero.
\end{proof}

\begin{Remark}\label{rmk:1172}
In the above proof, if $\mb v_1$ and $\mb v_2$ are real analytic,
then the function $\phi$ must be a real analytic function.
For a given real analytic function germ $\Psi$,  
a real analytic function germ $\psi$ satisfying $\psi^2=\Psi$ is uniquely determined up to
$\pm$-ambiguity whenever it exists. Thus,
if a real analytic curvature line flow $\mc F$ exists, 
there are no real analytic curvature line flows other than $\mc F$ and
$\mc F^\perp$. 
\end{Remark}

Even when a time-like surface is real analytic, 
the corresponding curvature line flows might not be real analytic:

\begin{Example}
When $\mu(\xi,\eta):=\xi^4+\eta^4-\xi^3\eta^2$, we have
(cf. \eqref{eq:1124})
$$
a(\xi,\eta)=-12\xi^2 \eta,\qquad
b(\xi,\eta)=
-2 \left(\xi^3-6 \xi^2+3 \xi \eta^2-6 \eta^2\right)
$$
and
$$
b(\xi,\eta)\pm a(\xi,\eta)=12 \Big(\xi^2 + \eta^2+o(\xi^2+\eta^2)\Big),
$$
where $o(\xi^2+\eta^2)$ is the Landau symbol, that is, 
a sum of terms higher than $\xi^2+\eta^2$.
So $b^2-a^2$ is non-negative for sufficiently small $|\xi|$ and $|\eta|$.
However, a continuous function $\phi$
satisfying $\phi^2=b^2-a^2$
is $C^1$-differentiable but not $C^2$-differentiable. 
So the resulting curvature line flows
$\mc F$ and $\mc F^\perp$ are $C^1$-differentiable.
\end{Example}

\rm
Let $h$ be a time-like surface whose principal curvatures 
are real-valued on a neighborhood of an isolated umbilic $o$.
By Theorem E,
the existence of a $C^1$-differentiable curvature line flow whose index is
equal to zero at the umbilic is guaranteed.
However, when quasi umbilics accumulate at $o$,
there might exist another 
curvature line flow whose index does not vanish at $o$
as the following example shows:

\begin{Exa}\label{eq:I2}
If we set
$
\mu(\xi,\eta):=\xi^2\eta^2+(\xi^4+\eta^4)/6,
$
then we have
$$
T_{\mu}(\xi,\eta)=2\pmt{2\xi\eta & \xi^2+\eta^2 \\
             \xi^2+\eta^2 & 2 \xi\eta}.
$$
Hence, the two vector fields
$$
\mb v_1:=-\xi\frac{\partial}{\partial \xi}+\eta\frac{\partial}{\partial \eta},
\qquad 
\mb v_2:=\eta\frac{\partial}{\partial \xi}-\xi\frac{\partial}{\partial \eta}
$$
are both real analytic null vector fields of $T_\mu$, and generate 
two curvature line flows $\mc F_1$ and $\mc F^\perp_1$
whose indices at $o$ are $-1$ and $1$, respectively
(see the left and the center of Figure \ref{fig:1360}).

\begin{figure}[htb]
 \begin{center}
\includegraphics[height=3.2cm]{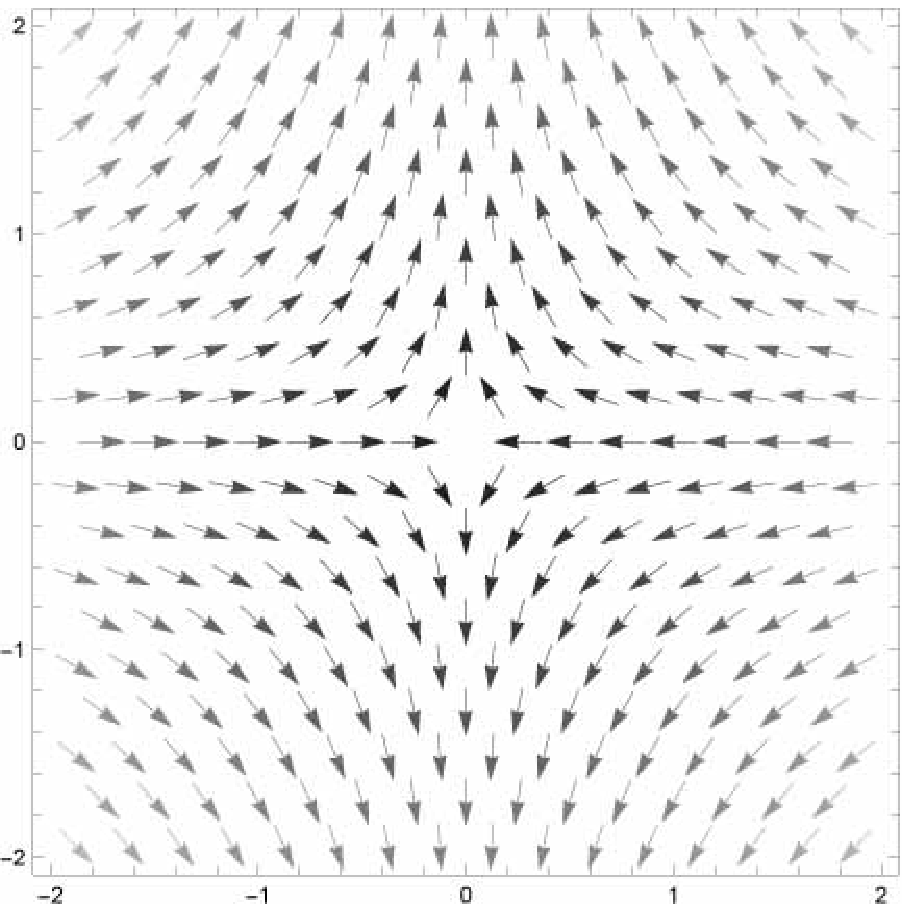}\qquad
\includegraphics[height=3.2cm]{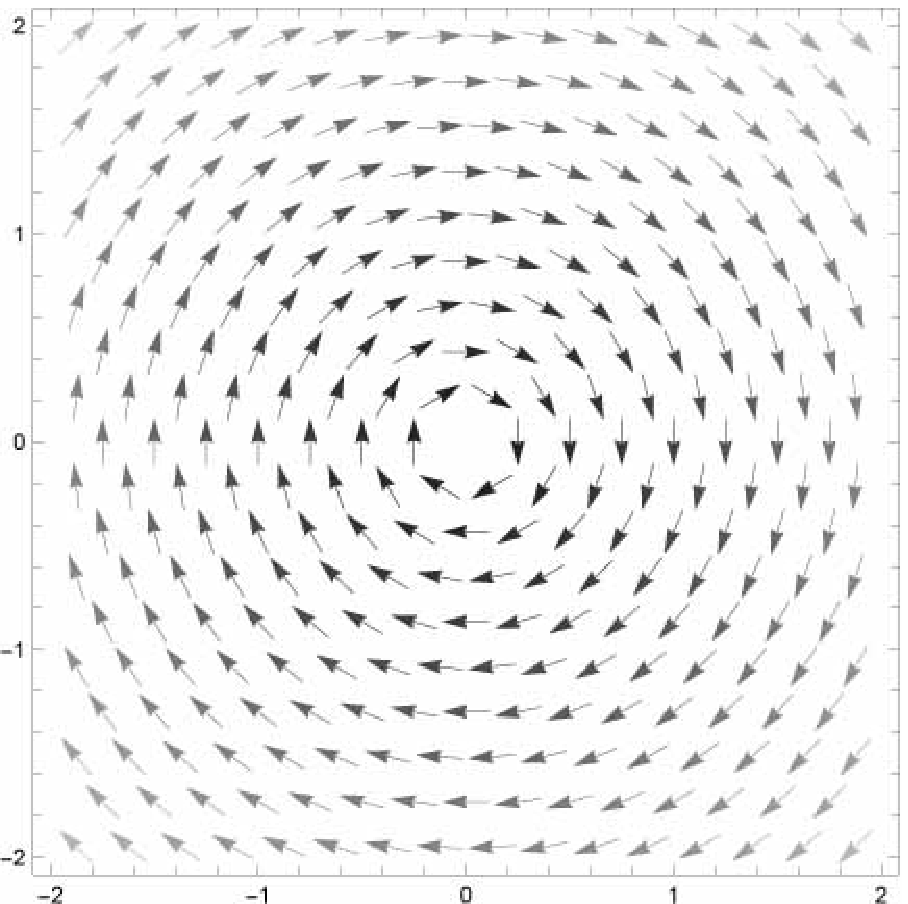}\qquad
\includegraphics[height=3.2cm]{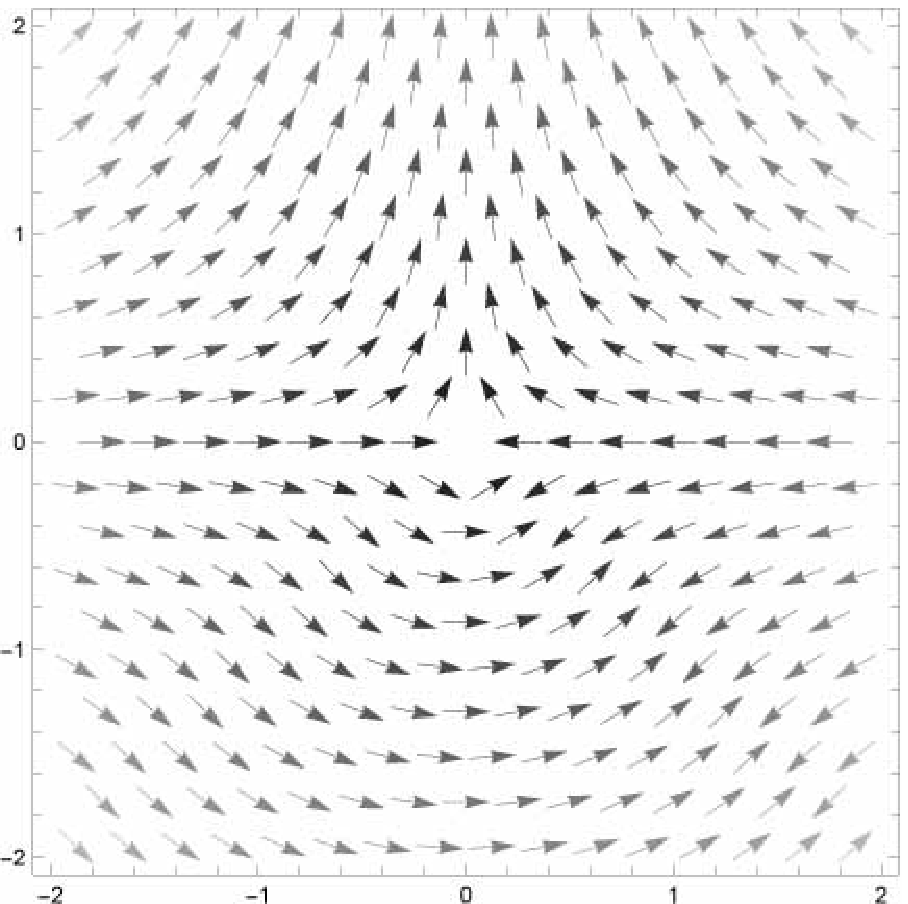}
\caption{The flows generated by $\mb v_1$ (left), 
generated by $\mb v_2$ (center) and the flow $\mathcal{F}_{1/2}$ (right)}
\label{fig:1360}
\end{center}
\end{figure}

On the other hand, 
we denote by $\mathcal{F}_{1/2}$ a $C^1$-differentiable curvature line flow 
induced by $[{\bf v}_2]$ 
if $\eta \leq -|\xi |$ and $[{\bf v}_1]$ if otherwise. 
Then $\mathcal{F}_{1/2}$ has index $-1/2$ at $o$
(see the right of Figure \ref{fig:1360}). 
Hence taking into account
the flows $\mathcal{F}_0$ and $\mathcal{F}^{\perp}_0$ given in the 
proof of Theorem E as well as $\mathcal{F}^{\perp}_{1/2}$, 
the time-like surface associated with 
$\mu$ admits curvature line flows whose indices
take values $0,\pm1/2$ and $\pm 1$.
\end{Exa}

As a generalization of this example, we prove Theorem F:

\begin{proof}[Proof of Theorem F]
We let $j$ denote the imaginary digit of the para-complex numbers, that is,
$j^2=1$ and
any para-complex number can be written in the form 
$a+jb$ ($a,b\in \R$).
The function $\mu$ in Example 
\ref{eq:I2}
can then be rewritten as
$\mu (\xi , \eta )=\op{Re} (\xi +j\eta)^4/6$.
As a generalization, we set
$
\mu_m (\xi , \eta ):=\op{ Re}(\xi +j\eta )^{m}
$
for each positive integer $m(\ge 3)$.
For the sake of simplicity, we set $\mu:=\mu_m$,
and denote by $h(\xi,\eta)$ the corresponding 
time-like surface (cf. \eqref{eq:h734}),
which can be considered as an example 
analogous to Example \ref{exa:Z402}.
Then
$$
T_{\mu} 
=m(m-1)
\pmt{
\op{Im}(\xi+j \eta)^{m-2} & \op{Re}(\xi+j \eta)^{m-2} \\
\op{Re}(\xi+j \eta)^{m-2} & \op{Im}(\xi+j \eta)^{m-2}
}
$$
holds. 
We set $N_2[a+jb]:=a^2-b^2$ ($a,b\in \R$), 
which plays an analogue of the square of the norm of a complex number 
for a para-complex number.
In fact, one can easily check the identity
$N_2[(a+jb)(c+jd)]=N_2[a+jb]\, N_2[c+jd]$
for $a,b,c,d\in \R$.
Using this,
we obtain
\begin{align*}
\frac{\det T_\mu}{m^2(m-1)^2}
&=
\Big(
(\op{Im}(\xi+j \eta)^{m-2})^2-(\op{Re}(\xi+j \eta)^{m-2})^2
\Big) \\
&=-N_2[(\xi+j \eta)^{m-2}]=-N_2[\xi+j \eta]^{m-2}=
-(\xi^2-\eta^2)^{m-2}.
\end{align*}
Moreover, we have
\begin{equation}\label{eq:1294}
\frac{\mu_{\xi\xi}\mu_{\eta\eta}-\mu_{\xi\eta}^2}{m^2(m-1)^2}=
\Big(\op{Re}(\xi+j \eta)^{m-2}\Big)^2-
\Big(\op{Im}(\xi+j \eta)^{m-2}\Big)^2
=
-\frac{\det T_\mu}{m^2(m-1)^2}.
\end{equation}
If $m$ is an odd integer, then 
$\det T_\mu$ takes positive values when $|\eta|>|\xi|$, and
near such a point $(\xi,\eta)$, the curvature line flows of $h$ do not exist. 

So, we set $m=2n$ ($n\ge 2$).
By \eqref{eq:1294}, the Gaussian curvature of $h$ is then non-positive.
To simplify notations,
we define 
two real-valued functions $\alpha$, $\beta$ of $\xi$, $\eta$
by
\begin{equation}\label{eq:alpha-beta}
\alpha+j \beta :=(\xi+j \eta)^{n-1}.
\end{equation}
Then we can write
$$
\frac{T_{\mu}}{m(m-1)} 
=
\pmt{
\op{Im}(\alpha+j \beta)^2 & \op{Re}(\alpha+j \beta)^2 \\
\op{Re}(\alpha+j \beta)^2 & \op{Im}(\alpha+j \beta)^2
}=
\pmt{
2\alpha\beta & \alpha^2+\beta^2 \\
\alpha^2+\beta^2 & 2\alpha\beta
},
$$
and
$$
\mb v(\xi,\eta):=-\alpha(\xi,\eta) \frac{\partial}{\partial \xi}
+\beta(\xi,\eta) \frac{\partial}{\partial \eta},
\quad
\mb w(\xi,\eta):=\beta(\xi,\eta) \frac{\partial}{\partial \xi}
-\alpha(\xi,\eta) \frac{\partial}{\partial \eta}
$$
yield a pair of $C^\infty$-differentiable  
null vector fields of $T_{\mu}$.
Obviously, the origin $o$ is an isolated zero of these vector fields.

We first consider the case that $n$ is odd and set $n=2k+1$ ($k\ge 1$).
Then $\alpha=\op{Re}(x+jy)^{2}$ with $x+jy:=(\xi+j \eta)^{k}$.
Since $\op{Re}(x+jy)^2=x^2+y^2$ for $x,y\in \R$,
the function $\alpha$ is non-negative, and 
the indices of $\mb v$ and $\mb w$
at $o$ are equal to zero.

We next consider the case that $n$ is even.
Each point $(\xi,\eta)\in \Lo^2$ can be identified with
the para-complex number $\xi+j\eta$. 
Motivated by the definition of $\alpha$ and $\beta$ (cf. \eqref{eq:alpha-beta}),
we set
$
\Phi:\Lo^2\ni (\xi,\eta)\mapsto (\xi+j \eta)^{n-1}\in \Lo^2.
$
Then, 
any space-like (resp. time-like) vector
of $\Lo^2$
can be written as 
$$
(-1)^\sigma r E(jt)\quad
(\text{resp.}\,\, (-1)^\sigma r j E(jt))
$$
where
$E(jt):=\cosh t +j \sinh t$, 
$\sigma\in \{0,1\}$, $r>0$ and $t\in \R$.
Since $n-1$ is odd, we have
\begin{align}\label{eq:1339}
&\Phi ((-1)^{\sigma} rj^{\tau} E(jt))=(-1)^{\sigma} r^{n-1} j^{\tau} E(j(n-1)t), \\
\nonumber
&\Phi ((-1)^{\sigma} r(1\pm j))= (-1)^{\sigma} 2^{n-1} r^{n-1} (1\pm j) \quad 
(t\in \mathbb{R}, \sigma, \tau \in \{ 0, 1\}). 
\end{align}
There are four sectors separated by the two lines $\xi=\pm \eta$ in $\Lo^2$.
From \eqref{eq:1339}, we observe that
$\Phi$ maps each of these sectors to itself 
as an orientation preserving diffeomorphism.
So the indices of $\mb v$ and $\mb w$ at $o$ are equal to $-1$ and $1$, respectively.
As this construction is applicable for each even
integer $n(\ge 2)$, we obtain Theorem~F.
\end{proof}

\begin{Rmk}
In this paper, we used modified Ribaucour's parametrizations 
to prove all of the assertions and to construct
examples.  Instead one may use the method of (orthonormal) moving frames
 except for Theorems A, F and Proposition B. 
In this case, the Hessian matrix of the function
associated with a modified Ribaucour's parametrization
corresponds to the second fundamental matrix 
associated with the moving frame method  (cf. Remark \ref{rmk:TcT}). 
However, when constructing umbilics with a given index, 
the method of our  modified Ribaucour's parametrizations 
will be simpler than that using moving frames, 
as demonstrated in the proofs of Theorems A and F.
\end{Rmk}

\begin{acknowledgements}
The second author thanks Professors Farid Tari
and Federico Sanchez-Bringas
for informative conversations 
during their stay in Japan in November 2022.
The authors thank Professors Udo 
Hertrich Jeromin, Wayne Rossman for valuable comments.
\end{acknowledgements}

\medskip
\noindent
{\bf Declaration.}
The authors state that there is no conflict of interest.

\medskip
\noindent
{\bf Data Availability Statement.}
Data sharing not applicable to this article as no datasets were 
generated or analyzed during the current study.

\appendix
\rm

\section{
Null-direction flows of traceless symmetric tensors and
eigen-flows of symmetric tensors 
}

Let $U$ be a neighborhood of the origin $o$ 
in the $xy$-plane, and 
$$
S=s_{11}dx^2+2s_{12}dxdy+s_{22}dy^2
$$
be a symmetric tensor field on $U$,
which can be identified with the matrix-valued
function expressed as
$S=(s_{ij})_{i,j=1,2}$ ($s_{21}:=s_{12}$).
Suppose that $o$ is an isolated scalar point of $S$.
Then, at each point $p\in U\setminus \{o\}$,
we may assume that 
there exists a pair of unit vectors $(\mb v_1,\mb v_2)$
satisfying $S\mb v_i=\lambda_i \mb v_i$ ($i=1,2$) 
and  $\lambda_1(p)\ne \lambda_2(p)$
(see \cite[\S 14]{UY}).
Although $\mb v_i$ ($i=1,2$) have $\pm$-ambiguity,
they define the elements $\xi_i:=[\mb v_i]$ ($i=1,2$)
in the projective space associated with the tangent space $T_pU$ 
at each $p\in U\setminus \{o\}$ (cf. \eqref{eq:1038}). 
Each of these two fields 
$\xi_1$ and $\xi_2$ is called an
{\it eigen-directional field} of $S$  (cf. \cite[\S 14]{UY}).
We remark that $\xi_1$ is perpendicular to $\xi_2$ at each point of $U\setminus \{o\}$.
We define the vector field
\begin{equation}\label{eq:ds931}
\mb v_S^{}:=d_1\frac{\partial}{\partial x}+d_2\frac{\partial}{\partial y}\qquad
(d_1:=s_{11}-s_{22},\,\, d_2:=2s_{12})
\end{equation}
on $U$, 
which is called
the {\it characteristic vector field} of $S$.
We remark that $o$ is an isolated zero point of ${\mb v}_S^{}$.
The following assertion is well-known:

\begin{Fact}[{\cite[Proposition 15.4]{UY}}]\label{fact:1128}
If $o$ is an isolated scalar point of $S$,
then the indices $\op{Ind}_o(\xi_i)$ $(i=1,2)$
of the eigen-directional fields $($i.e.~the indices of 
the eigen-flows$)$ at $o$ satisfy
$$
\op{Ind}_o(\xi_1)=\op{Ind}_o(\xi_2)=\frac1{2}\op{Ind}_o(\mb v_S^{})\in \frac12 \Z,
$$
where $\op{Ind}_o(\mb v_S^{})$ is the index of the vector 
field $\mb v_S^{}$ at $o$.
\end{Fact}

We next define ``null directions":

\begin{Definition}
A tangent vector $\mb v$ at $p\in U$ is called a {\it null vector}
if $S(\mb v,\mb v)=0$, 
and the 1-dimensional vector subspace spanned by such a null vector
is called a {\it null direction} of $S$  at $p$.
\end{Definition}

If $\det S>0$, then its null directions 
never exist. 
If $\det S<0$ (resp. $\det S=0$) 
at $p\in U$,
then there exists
a pair of null directions
(resp. a unique null direction) at $p$.
In particular, if 
$
s_{11}+s_{22}=0,
$
that is, $S$ is a {\it traceless symmetric tensor}, 
then $\det S\le 0$.
Moreover, these directions are represented as vectors
$
\mb v_1\pm \mb v_2
$,
where $\mb v_1$ and $\mb v_2$ are linearly independent
unit eigenvectors of $S$,
that is, $S$ has the two null directions at each point 
of $U\setminus \{o\}$,
which are obtained by
$45^\circ$-rotation of the eigen-directions if $S$ is a traceless tensor.
So, we obtain the following:

\begin{Proposition}\label{prop:989}
Each of the null direction flows of a traceless symmetric tensor $S$
is obtained by the  $45^\circ$-rotation
of an eigen-flow of $S$.
\end{Proposition}


\begin{thebibliography}{00}

\bibitem{AFU}
N. Ando, T. Fujiyama and M. Umehara, 
$C^1$-umbilics with arbitrarily high indices,
Pacific Journal of Mathematics
{\bf 288}, (2017)
dx.doi.org/10.2140/pjm.2017.288.1

\bibitem{BHM}
F. Burstall, U. Hertrich-Jeromin and M.~L.~Miro,
{\it Ribaucour coordinates},
 Beitr. Algebra Geom. {\bf 60} (2019), 39--55, 
https://doi.org/10.1007/s13366-018-0391-9

\bibitem{FX}
F. Fontenele and F. Xavier,
{\it The index of isolated umbilics on surfaces
of non-positive curvature},
L'Enseignement Math\'ematique
(2) {\bf 61} (2015), 139--149.


\bibitem{T}
F. Tari, {\it Umbilics of surfaces in the Minkowski 3-space}, 
J. Math. Soc. Japan, {\bf 65}
(2013), 723--731.

\bibitem{UY}
M.~Umehara and K.~Yamada,
{Differential Geometry of Curves and Surfaces}, 
World Scientific (2015).

\end{thebibliography}
\end{document}